\documentclass{article}
\usepackage[utf8]{inputenc}
\renewcommand{\(}{\left(}
\renewcommand{\)}{\right)}

\newcommand{\sgn}{{\rm sgn}}
\usepackage{amsmath,amsfonts,graphicx,amsthm}
\usepackage{mathtools}
\usepackage{braket}
\usepackage[letterpaper,%
			left=3cm,right=3cm,top=25mm,bottom=25mm,%
            headheight=3mm,headsep=3mm,%
           ]{geometry}

\usepackage[pagebackref=false,colorlinks,linkcolor=blue,citecolor=magenta,menucolor=cyan]{hyperref}
\PassOptionsToPackage{unicode}{hyperref}
\PassOptionsToPackage{naturalnames}{hyperref}
\usepackage[nottoc,numbib]{tocbibind}

\newtheorem{theorem}{Theorem}

\newtheorem{definition}[theorem]{Definition}
\newtheorem{lemma}[theorem]{Lemma}

\newtheorem{remark}[theorem]{Remark}

\newtheorem{conjecture}[theorem]{Conjecture}

\usepackage{graphicx}
\usepackage{subcaption}

\title{The Set of Orthogonal Tensor Trains}
\author{Pardis Semnani and Elina Robeva}
\date{\today}

\begin{document}

\maketitle

\begin{abstract}
    In this paper we study the set of tensors that admit a special type of decomposition called an orthogonal tensor train decomposition. Finding equations defining varieties of low-rank tensors is generally a hard problem, however, the set of orthogonally decomposable tensors is defined by appealing quadratic equations. The tensors we consider are an extension of orthogonally decomposable tensors. We show that they are defined by similar quadratic equations, as well as an interesting higher-degree additional equation.
\end{abstract}

\section{Introduction}With the emergence of big data, it is more and more often the case that information is recorded in the form of a tensor (or multi-dimensional array). The importance of decomposing such a tensor is (at least) twofold. First, it provides hidden information about the data at hand,
and second, having a concise decomposition of the tensor allows us to store it much more
efficiently. 
One of the biggest obstacles in dealing with tensors, however, is that decomposing
them is often computationally hard. For example, finding (the number of terms of) the  CP-decomposition~\cite{Hitchcock} of a general tensor is NP-hard~\cite{HL}, and the set of tensors of CP rank at most $r$ is not closed for any $r\geq 2$, making the low-rank approximation problem impossible in some instances~\cite{KoBa09}. It is also notoriously hard to describe all the defining equations of the set of tensors of rank at most $r$~\cite{Landsberg}.

Finding the CP-decomposition of special subclasses of tensors, however, can be done efficiently. For example, low-rank tensors can be decomposed via Jennrich's algorithm and their subfamily of orthogonally decomposable tensors can be decomposed via the slice method~\cite{kolda2015symmetric} or via the tensor power method~\cite{Anandkumar2014}. Furthermore, the set of orthogonally decomposable tensors of bounded rank is closed~\cite{BDHR}. In fact, it is a real algebraic variety defined by interpretable quadratic equations which represent symmetry in the contraction of the tensor $T$ with itself. The eigenvectors and singular vector tuples of orthogonally decomposable tensors can be found efficiently~\cite{Robeva, RobSei}, making the family of such tensors as appealing as the set of matrices. However, such tensors are very rare -- the rank of a general $n\times n\times\cdots\times n$ ($d$-times) tensor is $\mathcal O(n^{d-1})$, while that of an orthogonally decomposable one is $n$.

Tensor networks provide many other ways of decomposing tensors. They originate from quantum physics and are used to depict the structure of steady states of Hamiltonians of quantum systems~\cite{TNNutshell,quantum}. Many types of tensor network decompositions, like tensor trains~\cite{oseledets}, also known as matrix product states~\cite{quantum}, are used in machine learning to decompose data tensors in meaningful ways. The CP-decomposition of a tensor can also be represented by a tensor network whose underlying graph is a hypergraph~\cite{RobSei2}. While tensor network decompositions can represent various tensor structures, they are not unique~\cite{ye2019tensor}.

More recently, tensor train decompositions for which the tensors at the vertices are orthogonally decomposable have been studied~\cite{HalMulRob20}. It was shown that those can be decomposed efficiently, at least in the case of length-2 tensor trains. In this work we study the equations that define the set of tensors that can be decomposed this way. Finding the description via equations would help in determining which tensors can be expressed in this way, and if it is easy to approximate a given tensor by such a tensor. We compare the equations we get to the equations defining the set of orthogonally decomposable tensors. It turns out, the two are quite similar and structured.

The rest of this manuscript is organized as follows. In Section~\ref{s2} we provide background on orthogonally decomposable tensors as well as tensor networks. In Section~\ref{s3} we turn to orthogonally decomposable tensor trains of length 2 and present our findings.

\section{Background}\label{s2}
In this section we provide the necessary background on orthogonal tensor decomposition, as well as tensor networks.
\subsection{Tensors}
We begin by defining tensors and symmetric tensors.

A {\em tensor} $T$ is an $n_1\times\cdots\times n_d$ table with elements in a field, which for us will be the reals. The number $d$ of sides of the table is called the {\em order} of the tensor. The set of all tensors of size $n_1\times\cdots\times n_d$ with real entries is denoted by $\mathbb R^{n_1}\otimes\cdots\otimes \mathbb R^{n_d}$.

A tensor $T\in\mathbb R^{n_1}\otimes\cdots\otimes\mathbb R^{n_d}$ is a {\em rank-1 } tensor if there exist vectors $v^{(1)}\in\mathbb R^{n_1},\ldots, v^{(d)}\in\mathbb R^{n_d}$ such that
$$T = v^{(1)}\otimes\cdots\otimes v^{(d)}.$$
The {\em CP-decomposition}~\cite{Hitchcock} expresses a tensor $T\in\mathbb R^{n_1}\otimes\cdots\otimes\mathbb R^{n_d}$ as a sum of rank-1 terms:
$$T = \sum_{i=1}^rv_i^{(1)}\otimes\cdots\otimes v_i^{(d)}.$$

An $n\times\cdots \times n$ ($d$-times) tensor $T$ is {\em symmetric} if for any permutation $\sigma$ on $\{1,\ldots, d\}$, $T_{i_1\ldots i_d}= T_{i_{\sigma(1)}\ldots i_{\sigma(d)}}$. The set of such tensors is denoted by $S^d(\mathbb R^n)$.

A symmetric tensor $T\in S^d(\mathbb R^n)$ has rank-1 if it can be expressed as
$$T = \lambda v^{\otimes d},$$
i.e., as a scalar times the outer product of a vector with itself $d$ times. A {\em symmetric CP-decomposition} of a symmetric tensor $T\in S^d(\mathbb R^n)$ expresses $T$ as a sum of symmetric rank-1 terms:
$$T = \sum_{i=1}^r\lambda_i v_i^{\otimes d}.$$

\subsection{Tensor networks}
In this subsection we provide the necessary background on tensor networks and tensor trains.
Let  $T\in\mathbb{R}^{m_1}\otimes\cdots\otimes\mathbb R^{m_k}$ and $S\in\mathbb{R}^{n_1}\otimes\cdots\otimes\mathbb R^{n_\ell}$ be two tensors such that $m_i=n_j$. The {\em contraction} of the tensors $T$ and $S$ along their $i$th and $j$th modes is the tensor $R\in\mathbb{R}^{m_1}\otimes\cdots\otimes\mathbb R^{m_{i-1}}\otimes\mathbb R^{m_{i+1}}\otimes\cdots\otimes\mathbb R^{m_k}\otimes\mathbb R^{n_1}\otimes\cdots\otimes\mathbb R^{n_{j-1}}\otimes\mathbb R^{n_{j+1}}\otimes\cdots\otimes \mathbb R^{n_\ell}$ whose entries are
\[R_{a_1\ldots a_{i-1} a_{i+1}\ldots a_{k}b_1\ldots b_{j-1}b_{j+1}\ldots b_{\ell}}=\sum_{c=1}^{m_i}T_{a_1\ldots a_{i-1}c a_{i+1}\ldots a_{k}}S_{b_1\ldots b_{j-1}cb_{j+1}\ldots b_{\ell}},\]
where $a_1\in\Set{1,\ldots,m_1},\ldots,a_k\in\Set{1,\ldots,m_k},b_1\in\Set{1,\ldots,n_1},\ldots,b_\ell\in\Set{1,\ldots,n_\ell}$.

Starting with a number of tensors, and applying repeated contractions in a prescribed manner yields the notion of a tensor network. A tensor network can be represented by a graph in which each node corresponds to one of the tensors involved in forming the tensor network, and the number of the edges connected to each node is equal to the order of the tensor corresponding to that node.

\begin{definition}
Consider a 3-tuple $G = (V, E, D)$, where $V$ is a nonempty set, 
$E$ is a subset of $\Set{\Set{v_1,v_2} | v_1,v_2\in V, v_1\neq v_2}$, and $D$ is a subset of $\Set{\Set{v} | v\in V}$. The elements of $V$, $E\cup D$, and $D$ are called the {\em vertices}, {\em edges}, and {\em dangling edges} of $G$ respectively, and $G$ is called a graph. Each edge in $e\in E\cup D$ is assigned a positive integer $n_e$. Each vertex $v\in V$ is assigned a tensor $T_v\in\otimes_{v \in e \in E \cup D} \mathbb R^{n_e}$. The resulting {\em tensor network} state $T\in \otimes_{e\in D} \mathbb R^{n_e}$ is obtained by contracting the tensors $T_v$ along all edges $e\in E$.
\end{definition}

For instance, in Figure \ref{f1}, two order-3 tensors $T\in\mathbb{F}^{n_1\times n_2\times n_3}$ and $S\in\mathbb{F}^{n_3\times n_4\times n_5}$ are contracted along their 3rd and 1st modes  respectively to form an $n_1\times n_2\times n_4\times n_5$ tensor network, which is called a tensor train of length 2.

We define a {\em tensor train} of length $l$ to be a tensor network consisting of $l$ vertices arranged in a line, each of which has 3 adjacent edges (see Figure~\ref{f4}).

\begin{figure}[ht]
\begin{center}\includegraphics[width=0.9\linewidth]{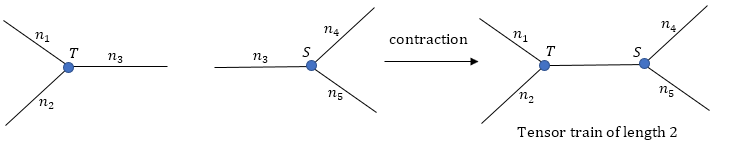}
\end{center}
\caption{How a tensor train of length 2 is generated}\label{f1}
\end{figure}
\begin{figure}[ht]
\centering
\includegraphics[scale=0.6]{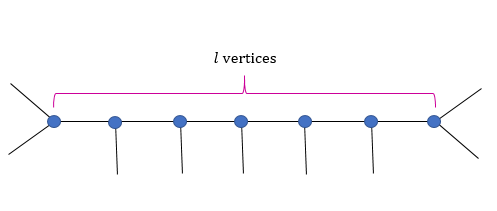}
\caption{Tensor train of length $l$}\label{f4}
\end{figure}

\begin{remark}
Matrix product states often appear in the condensed matter physics literature with one additional vertex and edge attached to each end of the diagram, as opposed to our diagrams as in Figure~\ref{f4} \cite{TNNutshell, quantum}. This format can be converted to the format introduced here simply with a contraction of the single edges, thus removing the vertices at each end. Conversely, given a tensor network in the form of our diagram, a multiplication by the identity matrix at each end will yield the matrix product state form. Thus, as the two network formats are  equivalent, we proceed with the use of the term ``tensor train'' to describe the networks we study in this paper. The reason for the modified form is to better generalize and extend the class of tensors known as orthogonally decomposable tensors, as studied in \cite{Anandkumar2014,BDHR,kolda2015symmetric,Robeva,RobSei}.
\end{remark}
\color{black}

\subsection{Orthogonally Decomposable Tensors}\label{ss2.3}
A tensor $T\in \mathbb R^{n_1}\otimes\cdots\otimes\mathbb R^{n_d}$ is {\em orthogonally decomposable} (or {\em odeco}) if it has a decomposition
$$T = \sum_{i=1}^rv_i^{(1)}\otimes\cdots\otimes v_i^{(d)},$$
where for every $j=1,\ldots, d$, the vectors $v_1^{(j)}, \ldots, v_r^{(j)}\in\mathbb R^{n_j}$ are orthogonal.

A symmetric tensor $T\in S^d(\mathbb R^n)$ is {\em symmetrically orthogonally decomposable} (or {\em symmetrically odeco}) if it has a decomposition
$$T = \sum_{i=1}^r\lambda_i v_i^{\otimes d},$$
where $v_1,\ldots, v_r\in\mathbb R^n$ are orthonormal.

It turns out that the set of odeco tensors (both in the symmetric and ordinary cases) is a real algebraic variety.

\begin{theorem}[\cite{BDHR}]\label{thm:1} Let $T\in\mathbb R^{n_1}\otimes\cdots\otimes \mathbb R^{n_d}$. For every $q=1,\ldots, d$, define $$T\bullet_qT\in \left(\mathbb R^{n_1}\otimes\cdots\otimes \mathbb R^{n_{q-1}}\otimes\mathbb R^{n_{q+1}}\otimes\cdots\otimes \mathbb R^{n_d}\right)^{\otimes 2}$$ by
$$(T\bullet_q T)_{i_1\ldots i_{q-1}i_{q+1}\ldots i_dj_1\ldots j_{q-1}j_{q+1}\ldots j_d} = \sum_{m=1}^{n_q}T_{i_1\ldots i_{q-1}mi_{q+1}\ldots i_d}T_{j_1\ldots j_{q-1}mj_{q+1}\ldots j_d},$$
i.e., it is the {\em contraction} of $T$ with itself along the $q$-th mode.
Then, $T$ is odeco if and only if for every $q=1,\ldots, d$, the tensor $T\bullet_q T$ lies in $S^2(\mathbb R^{n_1})\otimes\cdots\otimes S^2(\mathbb R^{n_d})$, which means that for all $\ell\in\Set{1,\ldots,d}$ with $\ell\neq q$, the entry $(T\bullet_q T)_{i_1\ldots i_{q-1}i_{q+1}\ldots i_dj_1\ldots j_{q-1}j_{q+1}\ldots j_d}$ remains unchanged if we permuted $i_\ell$ and $j_\ell$.
\end{theorem}
\vspace{0mm}
\begin{figure}[ht]
\begin{center}\includegraphics[width=0.9\linewidth]{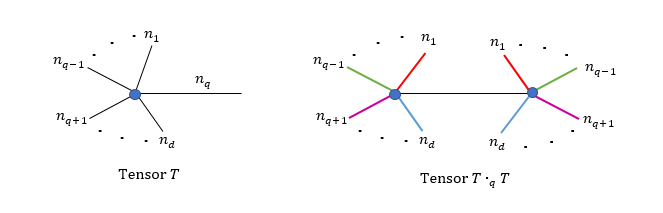}
\end{center}
\caption{$T$ is odeco iff the the entry $(T\bullet_q T)_{i_1\ldots i_{q-1}i_{q+1}\ldots i_dj_1\ldots j_{q-1}j_{q+1}\ldots j_d}$ remains unchanged when any two indices corresponding to edges with the same color are permuted.}\label{f2}
\end{figure}
A similar and easier to state result holds in the symmetric case.
\begin{theorem}[\cite{BDHR}]\label{thm:2} Let $T\in S^d(\mathbb R^n)$. Define $T\bullet_1 T\in S^{d-1}(\mathbb R^{n})\otimes S^{d-1}(\mathbb R^n)$ by
$$(T\bullet_1 T)_{i_2\ldots i_d j_2\ldots j_d} = \sum_{m=1}^n T_{mi_2\ldots i_d}T_{mj_2\ldots j_d},$$
i.e., the contraction of $T$ with itself along its first mode. Then, $T$ is odeco if and only if
$$T\bullet_1 T \in S^{2d-2}(\mathbb R^n),$$
i.e., the entry $(T\bullet_1 T)_{i_2\ldots i_d j_2\ldots j_d}$ remains unchanged after permuting any of its indices.
\end{theorem}
\vspace{0mm}
\begin{figure}[ht]
\begin{center}\includegraphics[width=0.8\linewidth]{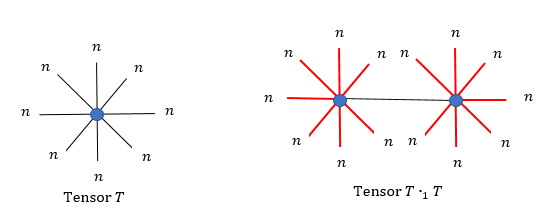}
\end{center}
\caption{$T$ is odeco iff the entry $(T\bullet_1 T)_{i_2\ldots i_d j_2\ldots j_d}$ remains unchanged when any of the indices are permuted.}\label{f3}
\end{figure}
In~\cite[Lemma 3.7]{Robeva}, it has been shown that the Zariski closure in $S^d(\mathbb{C}^n)$ of the set of symmetrically  orthogonally  decomposable tensors in $S^d(\mathbb{R}^n)$ is an irreducible component of the variety of the ideal defined by the symmetries (i.e., quadratic equations in the entries of $T$) introduced in Theorem \ref{thm:2}. So, it suffices to prove that this ideal is prime in order to conclude that it is exactly equal to the ideal of the Zariski closure in $S^d(\mathbb{C}^n)$ of the set of symmetrically  orthogonally  decomposable tensors in $S^d(\mathbb{R}^n)$. This has been proven when $n=2$ in~\cite[Theorem 3.6]{Robeva}, but is still an open problem for an arbitrary $n$.

\section{Orthogonally Decomposable Tensor Trains}\label{s3}

We now turn to tensors which decompose according to a tensor network, where all the tensors at the vertices of the network are symmetric and odeco. In particular, we confine our study to tensor trains. 

\begin{definition}
A tensor train of length $l$ (recall Figure~\ref{f4}) is said to be symmetrically orthogonally decomposable, or for short, symmetrically odeco, if each of the $l$ order-3 tensors generating the tensor train is $n\times n\times n$ symmetrically odeco. We denote the  set of such symmetrically odeco tensor trains by $SOT_{l,n}$. 
\end{definition}
It has been shown in~\cite{HalMulRob20} that if a tensor train of length 2 has a symmetrically odeco decomposition, one can find the decomposition efficiently.
Our goal here is to find polynomials which vanish on the set $SOT_{l,n}$ for certain $n,l$, and use them to determine the prime ideal of the Zariski closure of this set. In order to derive these polynomials for $l=2$, our approach is to consider the sets of variables
\begin{align*}
&P=\Set{p_{ijkl}|i,j,k,l\in\Set{1,\ldots,n}} && V=\Set{v_{ij}|i,j\in\Set{1,\ldots,n}} \\
& W=\Set{w_{ij}|i,j\in\Set{1,\ldots,n}}
&&M=\Set{\mu_i|i\in\Set{1,\ldots,n}} \\& L=\Set{\lambda_i|i\in\Set{1,\ldots,n}},
\end{align*}
define the following $\mathbb{R}$-algebra morphism
\[  \begin{array}{c}
\phi: \mathbb{R}\left[P\right]\to \mathbb{R}\left[ V\cup W\cup M\cup L   \right] \\
\phi\left(p_{ijkl}\right)=\sum_{a=1}^{n} \sum_{b=1}^{n} \sum_{c=1}^n \mu_a \lambda_b v_{a i} v_{a j} w_{b k} w_{b l}  v_{ac}w_{bc},
\end{array}
\]
and compute $\phi^{-1}\left(I\right)$ using computer algebra software such as Macaulay2, where
\[I=\left\langle \bigcup_{k=1}^n \Set{ \sum_{i=1}^n v_{ki}^2=1,\sum_{i=1}^n w_{ki}^2=1} \right\rangle +\left\langle \bigcup_{\substack{{k,l=1} \\ k\neq l}}^n \Set{\sum_{i=1}^n v_{ki}v_{li}=0,\sum_{i=1}^n w_{ki}w_{li}=0} \right\rangle.\]
We were unable to run an analogous program which would allow us to come up with the polynomials in the cases where $l>2$. So, we confine our study to length-2 tensor trains (see Figure~\ref{f1}) in this paper.

We now introduce one set of linear polynomials and one set of quadratic polynomials which vanish on $SOT_{2,n}
$, where $n\geq 2$.

Let $n\in\mathbb{N}$ and let $\mathcal{P}_n$ denote the set of polynomials
\[f_{a,b,c,d,\sigma_1,\sigma_2}\coloneqq p_{abcd}-p_{\sigma_1(a)\sigma_1(b)\sigma_2(c)\sigma_2(d)},\]
where $a,b,c,d\in\Set{1,\ldots,n}$, $\sigma_1$ is a permutation on $\Set{a,b}$, and $\sigma_2$ is a permutation on $\Set{c,d}$.
Moreover, let $\mathcal{Q}_n$ denote the set of polynomials
\[g^{(1)}_{a,b,c,d,e,f}\coloneqq\sum_{t=1}^n p_{abet}\ p_{cdft}-p_{abft}\ p_{cdet},\]
and 
\[g^{(2)}_{a,b,c,d,e,f}\coloneqq\sum_{t=1}^n p_{etab}\ p_{ftcd}-p_{ftab}\ p_{etcd},\]
where $a,b,c,d,e,f\in\Set{1,\ldots,n}$.

We now proceed to showing that the equations $\mathcal{P}_n$ and $\mathcal{Q}_n$ vanish on $SOT_{2,n}$. In particular, these also represent certain symmetries of the contracted tensor $T\bullet_qT$, similar to the odeco case described in section~\ref{ss2.3}. 
\begin{theorem} \label{thm:5}
Let $n\geq 2$. Then $\mathcal{P}_n$ and $\mathcal{Q}_n$ vanish on $SOT_{2,n}$.
\end{theorem}
\begin{figure}[h]
\begin{center}\includegraphics[width=0.9\linewidth]{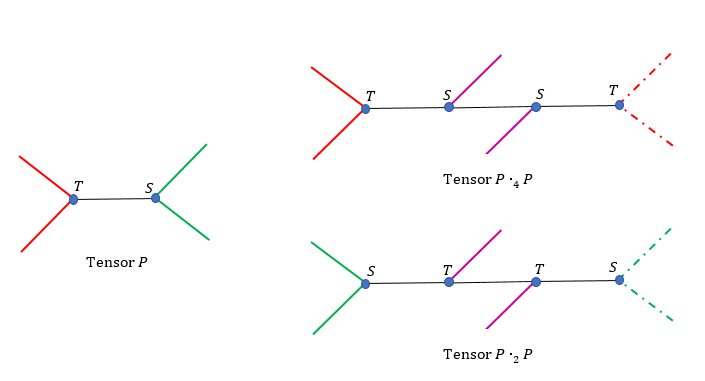}
\end{center}
\caption{The fact that $\mathcal{P}_n$ vanishes on $SOT_{2,n}$ shows that $P$ is symmetric with respect to any of the two pairs of modes indicated in red and green.\\The fact that $\mathcal{Q}_n$ vanishes on $SOT_{2,n}$ shows that $P\bullet_4 P$ and $P\bullet_2 P$ are symmetric with respect to the modes indicated in purple. }\label{f5}
\end{figure}
\begin{remark} The equations $\mathcal Q_n$ depict certain symmetries in the contracted tensors $P\bullet_2 P$ and $P\bullet_4 P$, very similar to the odeco case (see Theorems~\ref{thm:1} and \ref{thm:2}, and Figures~\ref{f2} and \ref{f3}). We have explained these symmetries in Figure~\ref{f5}. In particular, $P\bullet_4 P \in S^2(\mathbb R^n)\otimes S^2(\mathbb R^n)\otimes S^2(\mathbb R^n)$, where the first two $S^2(\mathbb R^n)$ correspond to the red edges and the red dashed edges respectively, and the last $S^2(\mathbb R^n)$ corresponds to the commuting of the indices coming from the purple edges and is the one that gives rise to half of the equations in $\mathcal Q_n$. Similarly, $P\bullet_2 P \in S^2(\mathbb R^n)\otimes S^2(\mathbb R^n)\otimes S^2(\mathbb R^n)$, where the first two $S^2(\mathbb R^n)$ correspond to the green and green dashed edges respectively, and the last $S^2(\mathbb R^n)$ corresponds to the fact that the indices along the purple edges commute and gives rise to the other half of the equations in $\mathcal Q_n$. 
\end{remark}
\begin{proof}
We begin by writing an arbitrary entry of a given tensor $P$ in $SOT_{2,n}$ in terms of the entries of the orthogonal vectors which appear in the decomposition of the two tensors at the nodes of $P$, and then proceed by writing the equations in $\mathcal{P}_n$ in terms of the entries of the two node tensors and the equations in $\mathcal{Q}_n$ in terms of the entries of these vectors. This is helpful since it provides us with the opportunity of applying the symmetry of the tensors and the orthogonality of the vectors.

Let $P\in SOT_{2,n}$. So, there exist vectors $v_1,\ldots,v_n,w_1,\ldots,w_n\in\mathbb{R}^n$ and scalars $\lambda_1,\ldots,\lambda_n $, $\mu_1,\ldots,\mu_n\in\mathbb{R}$ such that $\Set{v_1,\ldots,v_n}$ and $\Set{w_1,\ldots,w_n}$ are orthonormal sets, $T=\sum_{i=1}^n \lambda_i v_i^{\otimes 3}$, $S=\sum_{i=1}^n \mu_i w_i^{\otimes 3}$, and $P$ is the result of the contraction of $T$ and $S$ along their third modes. Therefore, for all $a,b,c,d\in\Set{1,\ldots,n}$,
{\allowdisplaybreaks
\begin{align*}
p_{abcd}&=\sum_{s=1}^n T_{abs} S_{cds}
=\sum_{s=1}^n \(\sum_{i=1}^n \lambda_i v_{ia}v_{ib}v_{is}\) \(\sum_{j=1}^n \mu_j w_{jc}w_{jd}w_{js}\)\\
&=\sum_{i=1}^n \sum_{j=1}^n \lambda_i \mu_j v_{ia} v_{ib} w_{jc} w_{jd} \(\sum_{s=1}^n v_{is} w_{js}\)
=\sum_{i=1}^n \sum_{j=1}^n \lambda_i \mu_j v_{ia} v_{ib} w_{jc} w_{jd} \langle v_i,w_j\rangle.
\end{align*}}
So, if $f_{a,b,c,d,\sigma_1,\sigma_2}\in \mathcal{P}_n$, then
{\allowdisplaybreaks
\begin{align*}
f_{a,b,c,d,\sigma_1,\sigma_2}&=p_{abcd}-p_{\sigma_1(a)\sigma_1(b)\sigma_2(c)\sigma_2(d)}\\
&=\(\sum_{s=1}^n T_{abs} S_{cds}\)  -\(\sum_{s=1}^n T_{\sigma_1(a)\sigma_1(b)s} S_{\sigma_2(c)\sigma_2(d)s}\) \\
&= \(\sum_{s=1}^n T_{abs} S_{cds}\)  -\(\sum_{s=1}^n T_{abs} S_{cds}\) &&  (\text{$T$ and $S$ are symmetric.})\\
&=0.
\end{align*}}
Thus, $\mathcal{P}_n$ vanishes on $SOT_{2,n}$. Now suppose that $g^{(1)}_{a,b,c,d,e,f}\in\mathcal{Q}_{n}$. Then 
{\allowdisplaybreaks
\begin{align*}
&g^{(1)}_{a,b,c,d,e,f}
=\sum_{t=1}^n p_{abet}\ p_{cdft}-p_{abft}\ p_{cdet}\\
&=\sum_{t=1}^n\biggm[ \(\sum_{i=1}^n \sum_{j=1}^n \lambda_i \mu_j v_{ia} v_{ib} w_{je} w_{jt} \langle v_i,w_j\rangle\) \( \sum_{k=1}^n \sum_{l=1}^n \lambda_k \mu_l v_{kc} v_{kd} w_{lf} w_{lt} \langle v_k,w_l\rangle \)  \\
&\phantom{\sum_{t=1}^n}-\(\sum_{i=1}^n \sum_{j=1}^n \lambda_i \mu_j v_{ia} v_{ib} w_{jf} w_{jt} \langle v_i,w_j\rangle\) \( \sum_{k=1}^n \sum_{l=1}^n \lambda_k \mu_l v_{kc} v_{kd} w_{le} w_{lt} \langle v_k,w_l\rangle \) \biggm]\\
&=\sum_{t=1}^n \biggm[ \( \sum_{i=1}^n \sum_{j=1}^n \sum_{k=1}^n \sum_{l=1}^n \lambda_i\lambda_k \mu_j \mu_l v_{ia} v_{ib} v_{kc} v_{kd} w_{je}w_{jt} w_{lf} w_{lt} \langle v_i , w_j\rangle \langle v_k,w_l\rangle     \)\\
&\phantom{\sum_{t=1}^n} - \( \sum_{i=1}^n \sum_{j=1}^n \sum_{k=1}^n \sum_{l=1}^n \lambda_i\lambda_k \mu_j \mu_l v_{ia} v_{ib} v_{kc} v_{kd} w_{jf}w_{jt} w_{le} w_{lt} \langle v_i , w_j\rangle \langle v_k,w_l\rangle     \)\biggm]\\
&=\sum_{t=1}^n\sum_{i=1}^n \sum_{j=1}^n \sum_{k=1}^n \sum_{l=1}^n \lambda_i\lambda_k \mu_j \mu_l v_{ia} v_{ib} v_{kc} v_{kd}w_{jt} w_{lt} \langle v_i , w_j\rangle \langle v_k,w_l\rangle\( w_{je}w_{lf}-w_{jf}w_{le}   \)\\
&=\sum_{i=1}^n \sum_{j=1}^n \sum_{k=1}^n \sum_{l=1}^n \lambda_i\lambda_k \mu_j \mu_l v_{ia} v_{ib} v_{kc} v_{kd}   \langle v_i , w_j\rangle \langle v_k,w_l\rangle\( w_{je}w_{lf}-w_{jf}w_{le}   \)  \( \sum_{t=1}^n w_{jt}w_{lt} \).
\end{align*}}
If $j=l$, then $w_{je}w_{lf}-w_{jf}w_{le}=0$, and otherwise, $\sum_{t=1}^n w_{jt}w_{lt}=\langle w_j, w_l\rangle=0$. This shows that  
$g^{(1)}_{a,b,c,d,e,f}=0$.
Similarly, if $g^{(2)}_{a,b,c,d,e,f}\in\mathcal{Q}_n$, then $g^{(2)}_{a,b,c,d,e,f}$ vanishes on $P$. Hence, $\mathcal{Q}_n$ vanishes on $SOT_{2,n}$.
\end{proof}

Now we proceed by introducing a polynomial of degree $n$ which vanishes on $SOT_{2,n}$.\\

Let $n\in\mathbb{N}$. Define
\[h_n\coloneqq \sum_{k_1=1}^n \ldots \sum_{k_n=1}^n \sum_{\sigma\in S_n}\sum_{\gamma\in S_n}{\rm sgn}(\sigma)\  {\rm sgn}(\gamma)\  p_{k_1 \sigma(1) k_n \gamma(1)}\ p_{k_2 \sigma(2) k_1 \gamma(2)}\cdots p_{k_n  \sigma(n) k_{n-1} \gamma(n)}.\]
\begin{theorem}
Let $n\geq 2$. Then $h_n$ vanishes on $SOT_{2,n}$.
\end{theorem}
\begin{proof}
Similar to the proof of Theorem~\ref{thm:5}, we first write the polynomial $h_n$ in terms of the entries of the vectors appearing in the orthogonal decomposition of the two tensors at the nodes of a given tensor $P$ in $SOT_{2,n}$, which yields equation~\eqref{q1}. When $n$ is even, we show that the terms of the right hand side polynomial in~\eqref{q1} can be split into two partitions such than the partitions cancel each other out. In the case when $n$ is odd, by manipulating the terms of the right hand side polynomial in~\eqref{q1} and applying the Parseval's Equation, we show that $h_n$ vanishes on $SOT_{2,n}$.

Let $P\in SOT_{2,n}$. So, there exist vectors $v_1,\ldots,v_n,w_1,\ldots,w_n\in\mathbb{R}^n$ and scalars $\lambda_1,\ldots,\lambda_n$, $\mu_1,\ldots,\mu_n\in\mathbb{R}$ such that $\Set{v_1,\ldots,v_n}$ and $\Set{w_1,\ldots,w_n}$ are orthonormal sets, $T=\sum_{i=1}^n \lambda_i v_i^{\otimes 3}$, $S=\sum_{i=1}^n \mu_i w_i^{\otimes 3}$, and $P$ is the result of the contraction of $T$ and $S$ along their third modes.
Note that
{\allowdisplaybreaks
\begin{align*}
    h_n&=\sum_{k_1=1}^n\ldots \sum_{k_n=1}^n \sum_{\sigma\in S_n}\sum_{\gamma \in S_n} {\rm sgn}(\sigma)\ {\rm sgn}(\gamma) \times\\
    &\Bigg[\(\sum_{i_1=1}^n \sum_{j_1=1}^n \lambda_{i_1} \mu_{j_1} \langle v_{i_1},w_{j_1}\rangle v_{i_1 k_1} v_{i_1 \sigma(1)} w_{j_1 k_n} w_{j_1 \gamma(1)}\)\times\\
    &\phantom{\Bigg[}\(\sum_{i_2=1}^n \sum_{j_2=1}^n \lambda_{i_2} \mu_{j_2} \langle v_{i_2},w_{j_2}\rangle v_{i_2 k_2} v_{i_2 \sigma(2)} w_{j_2 k_1} w_{j_2 \gamma(2)}\)\times\\
    &\phantom{\Bigg[\ \ \ }\vdots\\
    &\phantom{\Bigg[}\(\sum_{i_n=1}^n \sum_{j_n=1}^n \lambda_{i_n} \mu_{j_n} \langle v_{i_n},w_{j_n}\rangle v_{i_n k_n} v_{i_n \sigma(n)} w_{j_n k_{n-1}} w_{j_n \gamma(n)}\)\Bigg]\\
    &=\sum_{i_1=1}^n \sum_{j_1=1}^n \ldots \sum_{i_n=1}^n \sum_{j_n=1}^n \sum_{k_1=1}^n \ldots \sum_{k_n=1}^n \sum_{\sigma\in S_n}\sum_{\gamma\in S_n} {\rm sgn}(\sigma)\  {\rm sgn}(\gamma) 
    \lambda_{i_1}\mu_{j_1}\cdots \lambda_{i_n}\mu_{j_n} \\
    &\phantom{=}\langle v_{i_1},w_{j_1}\rangle \cdots \langle v_{i_n},w_{j_n}\rangle v_{i_1k_1}v_{i_2 k_2}\cdots v_{i_n k_n}w_{j_1 k_n} w_{j_2 k_1}\cdots w_{j_n k_{n-1}}
    v_{i_1\sigma(1)}\cdots v_{i_n \sigma(n)}\\
    &\phantom{=}w_{j_1\gamma(1)}\cdots w_{j_n\gamma(n)}\\
    &=\sum_{i_1=1}^n\sum_{j_1=1}^n \ldots \sum_{i_n=1}^n \sum_{j_n=1}^n  \lambda_{i_1}\mu_{j_1}\cdots \lambda_{i_n}\mu_{j_n}\langle v_{i_1},w_{j_1}\rangle \cdots \langle v_{i_n},w_{j_n}\rangle\\
    &\phantom{=}\(\sum_{k_1=1}^n v_{i_1 k_1}w_{j_2 k_1}\)\cdots \(\sum_{k_{n-1}=1}^n v_{i_{n-1} k_{n-1}}w_{j_n k_{n-1}}\) \(\sum_{k_n=1}^n v_{i_n k_n}w_{j_1 k_n}\)\\
    &\phantom{=}\(\sum_{\sigma \in S_n} {\rm sgn(\sigma)} v_{i_1 \sigma(1)}\cdots v_{i_n \sigma(n)}  \) \(\sum_{\gamma \in S_n} {\rm sgn(\gamma)} w_{j_1 \gamma(1)}\cdots w_{j_n \gamma(n)}  \)\\
    &=\sum_{i_1=1}^n \sum_{j_1=1}^n\ldots \sum_{i_n=1}^n \sum_{j_n=1}^n \lambda_{i_1}\mu_{j_1}\cdots \lambda_{i_n}\mu_{j_n}\langle v_{i_1},w_{j_1}\rangle\cdots \langle v_{i_{n-1}},w_{j_{n-1}}\rangle \langle v_{i_n}, w_{j_n}\rangle\\
    &\phantom{=}\langle v_{i_1},w_{j_2}\rangle \cdots \langle v_{i_{n-1}},w_{j_n}\rangle
    \langle v_{i_n},w_{j_1}\rangle
    \det \(V_{i_1\ldots i_n}\) \det \(W_{j_1\ldots j_n}\),
\end{align*}}
where $V_{i_1\ldots i_n}\coloneqq\left[ \begin{array}{ccc} - & v_{i_1} ^T - \\ &\vdots & \\ - & v_{i_n}^T&-\end{array}\right]$ and $W_{j_1\ldots j_n}\coloneqq\left[ \begin{array}{ccc} - & w_{j_1} ^T - \\ &\vdots & \\ - & w_{j_n}^T&-\end{array}\right].$\\
For all $\(i_1,\ldots,i_n\)\in\Set{1,\ldots,n}^n$, if there exists $a,b\in\Set{1,\ldots,n}$ such that $a\neq b$ and $i_a=i_b$, then $\det\(V_{i_1\ldots i_n}\)=0$. Similarly, $\det\(W_{j_1\ldots j_n}\)=0$ if there exists a pair of equal coordinates in $\(j_1,\ldots,j_n\)$. Thus,
\begin{align}\label{q1}
    h_n=\det\(V_{1\ldots n}\) \det\(W_{1\ldots n}\) \sum_{\sigma, \gamma\in S_n}&{\rm sgn}(\sigma\gamma) \langle v_{\sigma(1)},w_{\gamma(1)}\rangle \cdots \langle v_{\sigma(n)},w_{\gamma(n)}\rangle\\
    &\langle v_{\sigma(1)},w_{\gamma(2)}\rangle
    \cdots\langle
     v_{\sigma(n-1)},w_{\gamma(n)}\rangle
    \langle v_{\sigma(n)},w_{\gamma(1)}\rangle.\notag
\end{align}
Now assume that $n$ is even. Define 
\begin{align*}
    \begin{array}{cc}
    \begin{array}{c}   
    F_1:S_n \to S_n \\
    F_1 \(\left[i_1,i_2,\ldots,i_n  \right] \)=
    \left[ i_1,i_n,i_{n-1},\ldots,i_2 \right]
    \end{array} &
    \begin{array}{c}   
    F_2:S_n \to S_n \\
    F_2 \(\left[j_1,j_2,\ldots,j_n  \right]   \)=
   \left[j_2,j_1,j_n,j_{n-1},j_{n-2},\ldots,j_3  \right]
    \end{array}.
    \end{array}
\end{align*}
Since for all $\sigma=\left[i_1,\ldots,i_n\right]\in S_n$, $F_1(\sigma)$ is obtained from $\sigma$ by swapping $i_2$ and $i_n$, swapping $i_3$ and $i_{n-1}$, ..., and finally, swapping $i_{\frac{n}{2}}$ and $i_{\frac{n+4}{2}}$, $\sgn\(F_1(\sigma)\)=(-1)^{\frac{n}{2}-1}\sgn(\sigma)$.
Similarly, for all $\gamma=\left[j_1,\ldots,j_n\right]\in S_n$, $F_2(\gamma)$ is obtained from $\gamma$ by first swapping $j_1$ and $j_2$, and then, swapping $j_3$ and $j_n$, swapping $j_4$ and $j_{n-1}$, ..., and swapping $j_{\frac{n+2}{2}}$ and $i_{\frac{n+4}{2}}$. Thus, $\sgn\(F_2(\gamma)\)=(-1)^{\frac{n}{2}}\sgn(\gamma)$.
Hence,
\begin{align*}
    &{\rm sgn} \(F_1(\sigma)F_2(\gamma)\) \langle v_{F_1(\sigma)(1)},w_{F_2(\gamma)(1)}\rangle \cdots \langle v_{F_1(\sigma)(n)},w_{F_2(\gamma)(n)}\rangle
    \langle v_{F_1(\sigma)(1)},w_{F_2(\gamma)(2)}\rangle
    \cdots
    \langle v_{F_1(\sigma)(n)},w_{F_2(\gamma)(1)}\rangle\\
    =&(-1)^{\frac{n}{2}+\frac{n}{2}-1}{\rm sgn}(\sigma\gamma) \langle v_{i_1},w_{j_2} \rangle  \langle v_{i_n} ,w_{j_1} \rangle \langle v_{i_{n-1}},w_{j_n} \rangle \cdots \langle v_{i_2},w_{j_3} \rangle
    \langle v_{i_1},w_{j_1}\rangle \langle v_{i_n},w_{j_n}\rangle \cdots \langle v_{i_3},w_{j_3}\rangle \langle v_{i_2},w_{j_2}\rangle\\
    =&-{\rm sgn}(\sigma\gamma) \langle v_{\sigma(1)},w_{\gamma(1)}\rangle \cdots \langle v_{\sigma(n)},w_{\gamma(n)}\rangle
    \langle v_{\sigma(1)},w_{\gamma(2)}\rangle
    \cdots\rangle
    \langle v_{\sigma(n)},w_{\gamma(1)}\rangle.
\end{align*}
On the other hand, for all $\sigma\in S_n$, $F_1\left(F_1(\sigma)\right)=\sigma$ and $F_2\left(F_2(\sigma)\right)=\sigma$, which implies that both $F_1$ and $F_2$ are bijections. So,
\begin{align*}
   &2\( \sum_{\sigma, \gamma\in S_n}{\rm sgn}(\sigma\gamma) \langle v_{\sigma(1)},w_{\gamma(1)}\rangle \cdots \langle v_{\sigma(n)},w_{\gamma(n)}\rangle
    \langle v_{\sigma(1)},w_{\gamma(2)}\rangle
    \cdots\langle
     v_{\sigma(n-1)},w_{\gamma(n)}\rangle
    \langle v_{\sigma(n)},w_{\gamma(1)}\rangle\)\\
    =&\sum_{\sigma, \gamma\in S_n}{\rm sgn}(\sigma\gamma) \langle v_{\sigma(1)},w_{\gamma(1)}\rangle \cdots \langle v_{\sigma(n)},w_{\gamma(n)}\rangle
    \langle v_{\sigma(1)},w_{\gamma(2)}\rangle
    \cdots\langle
     v_{\sigma(n-1)},w_{\gamma(n)}\rangle
    \langle v_{\sigma(n)},w_{\gamma(1)}\rangle\\
    +& \sum_{\sigma, \gamma\in S_n}{\rm sgn}\(F_1(\sigma) F_2(\gamma)\) \langle v_{F_1(\sigma)(1)},w_{F_2(\gamma)(1)}\rangle \cdots \langle v_{F_1(\sigma)(n)},w_{F_2(\gamma)(n)}\rangle\\
    &\phantom{\sum_{\sigma, \gamma\in S_n}}
    \langle v_{F_1(\sigma)(1)},w_{F_2(\gamma)(2)}\rangle
    \cdots\langle
     v_{F_1(\sigma)(n-1)},w_{F_2(\gamma)(n)}\rangle
    \langle v_{F_1(\sigma)(n)},w_{F_2(\gamma)(1)}\rangle\\
    = &\sum_{\sigma, \gamma\in S_n} \bigg( {\rm sgn}(\sigma\gamma) \langle v_{\sigma(1)},w_{\gamma(1)}\rangle \cdots \langle v_{\sigma(n)},w_{\gamma(n)}\rangle
    \langle v_{\sigma(1)},w_{\gamma(2)}\rangle
    \cdots\langle
     v_{\sigma(n-1)},w_{\gamma(n)}\rangle
    \langle v_{\sigma(n)},w_{\gamma(1)}\rangle \\
    &\phantom{\sum_{\sigma, \gamma\in S_n}}+  {\rm sgn}\(F_1(\sigma) F_2(\gamma)\) \langle v_{F_1(\sigma)(1)},w_{F_2(\gamma)(1)}\rangle \cdots \langle v_{F_1(\sigma)(n)},w_{F_2(\gamma)(n)}\rangle\\
    &\phantom{\sum_{\sigma, \gamma\in S_n}}\langle v_{F_1(\sigma)(1)},w_{F_2(\gamma)(2)}\rangle
    \cdots\langle
     v_{F_1(\sigma)(n-1)},w_{F_2(\gamma)(n)}\rangle
    \langle v_{F_1(\sigma)(n)},w_{F_2(\gamma)(1)}\rangle \bigg)\\
    =& 0,
\end{align*}
and thus, 
\begin{align*}
    \sum_{\sigma, \gamma\in S_n}{\rm sgn}(\sigma\gamma) \langle v_{\sigma(1)},w_{\gamma(1)}\rangle \cdots \langle v_{\sigma(n)},w_{\gamma(n)}\rangle
    \langle v_{\sigma(1)},w_{\gamma(2)}\rangle
    \cdots\langle
     v_{\sigma(n-1)},w_{\gamma(n)}\rangle
    \langle v_{\sigma(n)},w_{\gamma(1)}\rangle=0.
\end{align*}
According to \eqref{q1}, this means that $h_n=0$. \\
We now proceed by assuming that $n$ is odd. In this case, for all $t\in\Set{1,\ldots,n}$, define
\begin{align*}
    \begin{array}{c}
    F_1^{(t)}:S_n\to S_n\\
    F_1^{(t)}\(\left[ i_1,\ldots,i_n\right] \)=\left[i_{\gamma^{-1}(t)-1}, i_{\gamma^{-1}(t)}, i_{\gamma^{-1}(t)+1},\ldots,i_n,i_1,i_2,\ldots,i_{\gamma^{-1}(t)-2} \right]
    \end{array}
\end{align*}
and
\begin{align*}
    \begin{array}{c}
    F_2^{(t)}:S_n\to S_n\\
    F_2^{(t)}\(\left[ j_1,\ldots,j_n\right] \)=\left[j_{\gamma^{-1}(t)-1}, j_{2}, j_{\gamma^{-1}(t)+1},\ldots,j_n,j_1,j_{\gamma^{-1}(t)},j_3,\ldots,j_{\gamma^{-1}(t)-2} \right]
    \end{array},
\end{align*}
where if $k<0$, by $i_k$ and $j_k$ we mean $i_l$ and $j_l$ respectively, such that $k \equiv l \mod{n}$ and $l\in\Set{1,\ldots,5}$.
Suppose that for all $\sigma=\left[i_1,\ldots,i_n  \right]\in S_n$, $F_1^{(t)}(\sigma)$ is obtained from $\sigma$ by swapping $m$ pairs. Then $\sgn\(F_1^{(t)}(\sigma)\)=(-1)^m \sgn(\sigma)$. Moreover, in this case, according to the definition of $F_2^{(t)}$, for all $\gamma=\left[j_1,\ldots,j_n\right]\in S_n$ with $\gamma(2)\neq t$, $F_2^{(t)}(\gamma)$ is obtained from $\gamma$ by first, swapping the same $m$ pairs and then, swapping $j_2$ and $j_{\gamma^{-1}(t)}$ in the resulting permutation, which means that $\sgn\(F_2^{(t)}(\gamma)\)=(-1)^{m+1} \sgn(\gamma)$.
So, if $\gamma(2)\neq t$, then 
\small{\begin{align*}
    &\sgn(F_1^{(t)}(\sigma) F_2^{(t)}({\gamma}) )\langle v_{F_1^{(t)}(\sigma)(1)},w_{F_2^{(t)}(\gamma)(1)} \rangle \langle v_{F_1^{(t)}(\sigma)(2)},w_{t} \rangle \langle v_{F_1^{(t)}(\sigma)a(3)},w_{F_2^{(t)}(\gamma)(3)} \rangle \cdots \langle v_{F_1^{(t)}(\sigma)(n)},w_{F_2^{(t)}(\gamma)(n)} \rangle \\ 
    &\phantom{\sgn (F_1^{(t)}(\sigma) F_1^{(t)}(\gamma))}\langle v_{F_1^{(t)}(\sigma)(1)},w_{t} \rangle  \langle v_{F_1^{(t)}(\sigma)(2)},w_{F_2^{(t)}(\gamma)(3)} \rangle  \langle v_{F_1^{(t)}(\sigma)(3)},w_{F_2^{(t)}(\gamma)(4)} \rangle \cdots  \langle v_{F_1^{(t)}(\sigma)(n)},w_{F_2^{(t)}(\gamma)(1)} \rangle \\
    =&(-1)^{2m+1}\sgn(\sigma \gamma) \langle v_{i_{\gamma^{-1}(t)-1}},w_{j_{\gamma^{-1}(t)-1}} \rangle \langle v_{i_{\gamma^{-1}(t)}},w_{t} \rangle \cdots \langle v_{i_n},w_{j_n} \rangle
    \langle v_{i_1},w_{j_1} \rangle \langle v_{i_2},w_{j_{\gamma^{-1}(t)}} \rangle \langle v_{i_{\gamma^{-1}(t)-2}},w_{j_{\gamma^{-1}(t)-2}} \rangle\\ 
    &\phantom{(-1)^{2m+1}\sgn(\sigma \gamma) }\langle v_{i_{\gamma^{-1}(t)-1}},w_{t} \rangle  \langle v_{i_{\gamma^{-1}(t)}},w_{j_{\gamma^{-1}(t)+1}}  \rangle \cdots \langle v_{i_n},w_{j_1} \rangle \langle v_{i_1},w_{j_{\gamma^{-1}(t)}} \rangle \langle v_{i_2},w_{j_3} \rangle \cdots  \langle v_{i_{\gamma^{-1}(t)-2}},w_{j_{\gamma^{-1}(t)-1}} \rangle\\
    =&-\sgn (\sigma \gamma)\langle v_{\sigma(1)},w_{\gamma(1)} \rangle \langle v_{\sigma(2)},w_{t} \rangle \langle v_{\sigma(3)},w_{\gamma(3)} \rangle \cdots \langle v_{\sigma(n)},w_{\gamma(n)} \rangle \\ 
    &\phantom{\sgn (\sigma \gamma)}\langle v_{\sigma(1)},w_{t} \rangle  \langle v_{\sigma(2)},w_{\gamma(3)} \rangle  \langle v_{\sigma(3)},w_{\gamma(4)} \rangle \cdots  \langle v_{\sigma(n)},w_{\gamma(1)} \rangle.
\end{align*}}

Now let $t\in\Set{1,\ldots,n}$. Since for all $\sigma\in S_n$, $F_1^{(t)}\(F_1^{(t)}(\sigma)\)=\sigma$ and $F_2^{(t)}\(F_2^{(t)}(\sigma)\)=\sigma$, $F_1^{(t)}$ and $F_2^{(t)}$ are bijections. Additionally, considering that for $\gamma\in S_n$, if $\gamma(2)\neq t$, then $F_2^{(t)}(\gamma)(2)\neq t$, $F_2^{(t)}$ is also a bijection on $A^{(t)}\coloneqq \Set{\gamma\in S_n | \gamma(2)\neq t}$. Hence,
\begin{align*}
    &2\sum_{\sigma\in S_n}\sum_{\gamma\in A^{(t)}}\sgn (\sigma \gamma)\langle v_{\sigma(1)},w_{\gamma(1)} \rangle \langle v_{\sigma(2)},w_{t} \rangle \langle v_{\sigma(3)},w_{\gamma(3)} \rangle \cdots \langle v_{\sigma(n)},w_{\gamma(n)} \rangle \\ 
    &\phantom{\sum_{\sigma\in S_n}\sum_{\gamma\in A^{(t)}}}\langle v_{\sigma(1)},w_{t} \rangle  \langle v_{\sigma(2)},w_{\gamma(3)} \rangle  \langle v_{\sigma(3)},w_{\gamma(4)} \rangle \cdots  \langle v_{\sigma(n)},w_{\gamma(1)} \rangle\\
    =&\sum_{\sigma\in S_n}\sum_{\gamma\in A^{(t)}}\sgn (\sigma \gamma)\langle v_{\sigma(1)},w_{\gamma(1)} \rangle \langle v_{\sigma(2)},w_{t} \rangle \langle v_{\sigma(3)},w_{\gamma(3)} \rangle \cdots \langle v_{\sigma(n)},w_{\gamma(n)} \rangle \\ 
    &\phantom{\sum_{\sigma\in S_n}\sum_{\gamma\in A^{(t)}}}\langle v_{\sigma(1)},w_{t} \rangle  \langle v_{\sigma(2)},w_{\gamma(3)} \rangle  \langle v_{\sigma(3)},w_{\gamma(4)} \rangle \cdots  \langle v_{\sigma(n)},w_{\gamma(1)} \rangle\\
    +&\sum_{\sigma\in S_n}\sum_{\gamma\in A^{(t)}}\sgn (F_1^{(t)}(\sigma) F_2^{(t)})\langle v_{F_1^{(t)}(\sigma)(1)},w_{F_2^{(t)}(1)} \rangle \langle v_{F_1^{(t)}(\sigma)(2)},w_{t} \rangle \langle v_{F_1^{(t)}(\sigma)(3)},w_{F_2^{(t)}(3)} \rangle \cdots \langle v_{F_1^{(t)}(\sigma)(n)},w_{F_2^{(t)}(n)} \rangle \\ 
    &\phantom{\sum_{\sigma\in S_n}\sum_{\gamma\in A^{(t)}}}\langle v_{F_1^{(t)}(\sigma)(1)},w_{t} \rangle  \langle v_{F_1^{(t)}(\sigma)(2)},w_{F_2^{(t)}(3)} \rangle  \langle v_{F_1^{(t)}(\sigma)(3)},w_{F_2^{(t)}(4)} \rangle \cdots  \langle v_{F_1^{(t)}(\sigma)(n)},w_{F_2^{(t)}(1)} \rangle\\
    =&\sum_{\sigma\in S_n}\sum_{\gamma\in A^{(t)}} \bigg(\sgn (\sigma \gamma)\langle v_{\sigma(1)},w_{\gamma(1)} \rangle \langle v_{\sigma(2)},w_{t} \rangle \langle v_{\sigma(3)},w_{\gamma(3)} \rangle \cdots \langle v_{\sigma(n)},w_{\gamma(n)} \rangle \\ 
    &\phantom{\sum_{\sigma\in S_n}\sum_{\gamma\in A^{(t)}}}\langle v_{\sigma(1)},w_{t} \rangle  \langle v_{\sigma(2)},w_{\gamma(3)} \rangle  \langle v_{\sigma(3)},w_{\gamma(4)} \rangle \cdots  \langle v_{\sigma(n)},w_{\gamma(1)} \rangle\\
    &\phantom{\sum_{\sigma\in S_n}\sum_{\gamma\in A^{(t)}}}+\sgn (F_1^{(t)}(\sigma) F_2^{(t)})\langle v_{F_1^{(t)}(\sigma)(1)},w_{F_2^{(t)}(1)} \rangle \langle v_{F_1^{(t)}(\sigma)(2)},w_{t} \rangle \langle v_{F_1^{(t)}(\sigma)(3)},w_{F_2^{(t)}(3)} \rangle \cdots \langle v_{F_1^{(t)}(\sigma)(n)},w_{F_2^{(t)}(n)} \rangle \\ 
    &\phantom{\sum_{\sigma\in S_n}\sum_{\gamma\in A^{(t)}}}\langle v_{F_1^{(t)}(\sigma)(1)},w_{t} \rangle  \langle v_{F_1^{(t)}(\sigma)(2)},w_{F_2^{(t)}(3)} \rangle  \langle v_{F_1^{(t)}(\sigma)(3)},w_{F_2^{(t)}(4)} \rangle \cdots  \langle v_{F_1^{(t)}(\sigma)(n)},w_{F_2^{(t)}(1)} \rangle\bigg)\\
    =& 0.
\end{align*}
Hence,
\allowdisplaybreaks{
\begin{align*}
    &\sum_{\sigma\in S_n}\sum_{\gamma\in S_n}\sgn (\sigma \gamma)\langle v_{\sigma(1)},w_{\gamma(1)} \rangle \langle v_{\sigma(2)},w_{t} \rangle \langle v_{\sigma(3)},w_{\gamma(3)} \rangle \cdots \langle v_{\sigma(n)},w_{\gamma(n)} \rangle \\ 
    &\phantom{\sum_{\sigma\in S_n}\sum_{\gamma\in S_n}}\langle v_{\sigma(1)},w_{t} \rangle  \langle v_{\sigma(2)},w_{\gamma(3)} \rangle  \langle v_{\sigma(3)},w_{\gamma(4)} \rangle \cdots  \langle v_{\sigma(n)},w_{\gamma(1)} \rangle\\
    =& \sum_{\sigma\in S_n}\sum_{\gamma\in S_n\setminus A^{(t)}}\sgn (\sigma \gamma)\langle v_{\sigma(1)},w_{\gamma(1)} \rangle \langle v_{\sigma(2)},w_{t} \rangle \langle v_{\sigma(3)},w_{\gamma(3)} \rangle \cdots \langle v_{\sigma(n)},w_{\gamma(n)} \rangle \\ 
    &\phantom{\sum_{\sigma\in S_n}\sum_{\gamma\in S_n\setminus A^{(t)}}}\langle v_{\sigma(1)},w_{t} \rangle  \langle v_{\sigma(2)},w_{\gamma(3)} \rangle  \langle v_{\sigma(3)},w_{\gamma(4)} \rangle \cdots  \langle v_{\sigma(n)},w_{\gamma(1)} \rangle\\
    =& \sum_{\sigma\in S_n}\sum_{\gamma\in S_n\setminus A^{(t)}}\sgn (\sigma \gamma)\langle v_{\sigma(1)},w_{\gamma(1)} \rangle \langle v_{\sigma(2)},w_{\gamma(2)} \rangle \langle v_{\sigma(3)},w_{\gamma(3)} \rangle \cdots \langle v_{\sigma(n)},w_{\gamma(n)} \rangle \\ 
    &\phantom{\sum_{\sigma\in S_n}\sum_{\gamma\in S_n\setminus A^{(t)}}}\langle v_{\sigma(1)},w_{\gamma(2)} \rangle  \langle v_{\sigma(2)},w_{\gamma(3)} \rangle  \langle v_{\sigma(3)},w_{\gamma(4)} \rangle \cdots  \langle v_{\sigma(n)},w_{\gamma(1)} \rangle.
\end{align*}}
This means that 
\begin{align*}
    &\sum_{t=1}^n\sum_{\sigma\in S_n}\sum_{\gamma\in S_n}\sgn (\sigma \gamma)\langle v_{\sigma(1)},w_{\gamma(1)} \rangle \langle v_{\sigma(2)},w_{t} \rangle \langle v_{\sigma(3)},w_{\gamma(3)} \rangle \cdots \langle v_{\sigma(n)},w_{\gamma(n)} \rangle \\ 
    &\phantom{\sum_{t=1}^n\sum_{\sigma\in S_n}\sum_{\gamma\in S_n}}\langle v_{\sigma(1)},w_{t} \rangle  \langle v_{\sigma(2)},w_{\gamma(3)} \rangle  \langle v_{\sigma(3)},w_{\gamma(4)} \rangle \cdots  \langle v_{\sigma(n)},w_{\gamma(1)} \rangle\\
    =& \sum_{t=1}^n\sum_{\sigma\in S_n}\sum_{\gamma\in S_n\setminus A^{(t)}}\sgn (\sigma \gamma)\langle v_{\sigma(1)},w_{\gamma(1)} \rangle \langle v_{\sigma(2)},w_{\gamma(2)} \rangle \langle v_{\sigma(3)},w_{\gamma(3)} \rangle \cdots \langle v_{\sigma(n)},w_{\gamma(n)} \rangle \\ 
    &\phantom{\sum_{t=1}^n\sum_{\sigma\in S_n}\sum_{\gamma\in S_n\setminus A^{(t)}}}\langle v_{\sigma(1)},w_{\gamma(2)} \rangle  \langle v_{\sigma(2)},w_{\gamma(3)} \rangle  \langle v_{\sigma(3)},w_{\gamma(4)} \rangle \cdots  \langle v_{\sigma(n)},w_{\gamma(1)} \rangle\\
    =& \sum_{\sigma\in S_n}\sum_{\gamma\in S_n}\sgn (\sigma \gamma)\langle v_{\sigma(1)},w_{\gamma(1)} \rangle \langle v_{\sigma(2)},w_{\gamma(2)} \rangle \langle v_{\sigma(3)},w_{\gamma(3)} \rangle \cdots \langle v_{\sigma(n)},w_{\gamma(n)} \rangle \\ 
    &\phantom{\sum_{\sigma\in S_n}\sum_{\gamma\in S_n}}\langle v_{\sigma(1)},w_{\gamma(2)} \rangle  \langle v_{\sigma(2)},w_{\gamma(3)} \rangle  \langle v_{\sigma(3)},w_{\gamma(4)} \rangle \cdots  \langle v_{\sigma(n)},w_{\gamma(1)} \rangle.
\end{align*}
On the other hand,
\begin{align*}
    &\sum_{t=1}^n\sum_{\sigma\in S_n}\sum_{\gamma\in S_n}\sgn (\sigma \gamma)\langle v_{\sigma(1)},w_{\gamma(1)} \rangle \langle v_{\sigma(2)},w_{t} \rangle \langle v_{\sigma(3)},w_{\gamma(3)} \rangle \cdots \langle v_{\sigma(n)},w_{\gamma(n)} \rangle \\ 
    &\phantom{\sum_{t=1}^n\sum_{\sigma\in S_n}\sum_{\gamma\in S_n}}\langle v_{\sigma(1)},w_{t} \rangle  \langle v_{\sigma(2)},w_{\gamma(3)} \rangle  \langle v_{\sigma(3)},w_{\gamma(4)} \rangle \cdots  \langle v_{\sigma(n)},w_{\gamma(1)} \rangle\\
    =& \sum_{\sigma\in S_n}\sum_{\gamma\in S_n}\sgn (\sigma \gamma) \langle v_{\sigma(1)},w_{\gamma(1)} \rangle  \langle v_{\sigma(3)},w_{\gamma(3)} \rangle \cdots \langle v_{\sigma(n)},w_{\gamma(n)} \rangle \\ 
    &\phantom{\sum_{\sigma\in S_n}\sum_{\gamma\in S_n}}\langle v_{\sigma(2)},w_{\gamma(3)} \rangle  \langle v_{\sigma(3)},w_{\gamma(4)} \rangle \cdots  \langle v_{\sigma(n)},w_{\gamma(1)} \rangle
    \bigg(\sum_{t=1}^n \langle v_{\sigma(2)},w_{t} \rangle\langle v_{\sigma(1)},w_{t} \rangle \bigg)\\
    =& \sum_{\sigma\in S_n}\sum_{\gamma\in S_n}\sgn (\sigma \gamma) \langle v_{\sigma(1)},w_{\gamma(1)} \rangle  \langle v_{\sigma(3)},w_{\gamma(3)} \rangle \cdots \langle v_{\sigma(n)},w_{\gamma(n)} \rangle \\ 
    &\phantom{\sum_{\sigma\in S_n}\sum_{\gamma\in S_n}}\langle v_{\sigma(2)},w_{\gamma(3)} \rangle  \langle v_{\sigma(3)},w_{\gamma(4)} \rangle \cdots  \langle v_{\sigma(n)},w_{\gamma(1)} \rangle
    \langle v_{\sigma(1)},v_{\sigma(2)}\rangle\\
    =&0.
\end{align*}
Thus, by \eqref{q1}, $h_n=0$.
\end{proof}

\begin{remark}
Note that the second part of the above proof, which proves the theorem when $n$ is odd, is actually applicable to the case where $n$ is even as well. However, the approach that has been used for the case where $n$ is even indicates that in this case, $h_n$ vanishes on any tensor train of length 2 whose nodes are tensors in $S^3\(\mathbb{R}^n\)$ with rank at most $n$, regardless of whether or not these tensors are orthogonally decomposable.
\end{remark}

When $n=2$ and $n=3$, we are able to show using Maple that $h_n$ does not lie in the ideal generated by $\mathcal{P}_n \cup \mathcal{Q}_n$, but vanishes on the set $SOT_{2,n}$.

Below we state our general conjecture about $SOT_{2,n}$ and provide some evidence for this conjecture.
\color{black}
\vspace{-0.5cm}
\begin{conjecture}
The Zariski closure of the set $SOT_{2,n}$ in $\mathbb{C}^{n^4}$ is cut out by the vanishing of $\mathcal P_n, \mathcal Q_n$, and the polynomial $h_n$. The ideal defined by these equations is prime.
\end{conjecture}
\begin{lemma}\label{lemma:1}
The dimension of the Zariski closure of $SOT_{2,n}$ in $\mathbb{C}^{n^4}$ is $n(n+1)-1$.
\end{lemma}
\begin{proof}
Define the map
\begin{align*}
    \begin{array}{c}
    \phi: \mathbb{R}^n \times SO_n \times \mathbb{R}^n \times SO_n \to \mathbb{R}^{n^4}\subseteq \mathbb{C}^{n^4}\\
    \(\lambda_1,\ldots,\lambda_n\),V,\(\mu_1,\ldots,\mu_n\),W \mapsto \sum_{i=1}^n \sum_{j=1}^n \lambda_i \mu_j v_i^{\otimes 2} \otimes w_j^{\otimes 2} \langle v_i, w_j \rangle
    \end{array},
\end{align*}
where $v_i$ and $w_i$ are the $i$th columns of the orthogonal matrices $V$ and $W$ respectively. Then $SOT_{2,n}={\rm Im}(\phi)$. It is proved in~\cite{HalMulRob20} that the decomposition of a generic tensor in $SOT_{2,n}$ can be found uniquely up to permutations and a global rescaling of $\lambda_i$'s and $\mu_j$'s (i.e. if we multiply all $\lambda_i$'s by $C$, we have to divide all $\mu_j$'s by $C$.). So, the dimension of the general fibre of $\phi$ is 1. Thus,
\begin{align*}
    \dim\({\rm Im}(\phi)\)=\dim\(\mathbb{R}^n \times SO_n \times \mathbb{R}^n \times SO_n\)-1 = 2n + 2\binom{n}{2} -1 = n(n+1)-1.
\end{align*}
So, the dimension of the Zariski closure of $SOT_{2,n}$ in $\mathbb{C}^{n^4}$ is $\dim\(\overline{{\rm Im} (\phi)}\) = n(n+1)-1$.
\end{proof}
When $n=2$, Lemma~\ref{lemma:1} shows that the dimension of the Zariski closure of the set $SOT_{2,2}$ in $\mathbb{C}^{16}$ equals 5. The dimension of the ideal $I \coloneqq \langle  \mathcal P_2, \mathcal Q_2 , h_2 \rangle$ also equals 5, which was checked using 
 \verb|Groebner[HilbertDimension]|
in Maple. Furthermore, \verb|PolynomialIdeals[PrimeDecomposition]| only generates one ideal when applied to $I$, which means that the radical of $I$ is prime. On the other hand, \verb|PolynomialIdeals[IsRadical]| verifies that $I$ is radical. So, $I$ is actually prime, and the above-mentioned equality of the dimensions proves that $I$ is the prime ideal of the Zariski closure of $SOT_{2,2}$ in $\mathbb{C}^{16}$.
We have not been able to verify that that the ideal generated by $\mathcal{P}_n$, $\mathcal{Q}_n$, and $h_n$ is prime for higher values of $n$ using Maple due to computational limitations.

\section{Acknowledgements}  ER was supported by an NSERC  Discovery Grant (DGECR-2020-00338).
\bibliographystyle{plain} \pagestyle{empty}
\bibliography{OrthogonalTT}

\end{document}